\documentclass[12pt]{article}
\usepackage{amssymb,amsmath,amsthm,amscd,latexsym}
\usepackage{float}
\usepackage{amsfonts}
\usepackage[mathscr]{eucal}
\usepackage{graphicx, color}
 \usepackage{tikz}
 \usepackage{url,hyperref}

\newtheorem{thm}{Theorem}[section]
\newtheorem{prop}[thm]{Proposition}
\newtheorem{lem}[thm]{Lemma}
\newtheorem{cor}[thm]{Corollary}

\newtheorem{defi}{Definition}

\font\msbm=msbm10 at 12pt

\newcommand{\FF}{\mbox{\msbm F}}

\newcommand{\F}{\mathbb{F}}
\newcommand{\w}{\omega}

\def\vv{\mathbf{v}}

\def\vw{\mathbf{w}}

\def\1v{\mathbf{1}}
\def\0v{\mathbf{0}}

\newcommand{\ds}{\displaystyle}

\begin{document}
\title{A Group Induced Four-circulant Construction for Self-dual Codes and New Extremal Binary Self-dual Codes}
\author{
Joe Gildea,\\
University of Chester \\
Department of Mathematics \\
Chester, UK \\
Abidin Kaya \\
Sampoerna University\\ 12780, Jakarta, Indonesia\\
Alexander Tylyshchak\footnote{This author's visit was partially supported by the London Mathematical Society (International Short Visits - Scheme 5).}\\
Department of Algebra\\
Uzhgorod National University\\
Uzhgorod, Ukraine\\
Bahattin Yildiz\footnote{Corresponding author} \\
Department of Mathematics \& Statistics\\
Northern Arizona University  \\
Flagstaff, AZ 86001, USA\\
e-mail: bahattin.yildiz@nau.edu
}
\date{}
\maketitle

\begin{abstract}
We introduce an altered version of the four circulant construction over group rings for self-dual codes. We consider this construction
over the binary field, the rings $\FF_2+u\FF_2$ and $\FF_4+u\FF_4$; using groups of order $3$, $7$, $9$, $13$, and $15$. Through these constructions and their extensions, we find binary self-dual codes of lengths $32$, $40$, $56$, $64$, $68$ and $80$, all of which are extremal or optimal.  In particular, we find five new self-dual codes of parameters $[56, 28,10]$, twenty-three extremal binary self-dual codes of length 68 with new weight enumerators and fifteen new self-dual codes of parameters $[80,40,14]$.
\end{abstract}



\section{Introduction}
Constructions coming from group rings that have emerged in recent works have extended the tools for binary self-dual codes. Many of the classical constructions have been found to be related to certain groups and considering different group rings have brought new construction methods into the literature. This connection was first highlighted in \cite{GKTY}, where a connection between certain group ring elements called unitary units and self-dual codes was established and the connection was used to produce many self-dual codes. The main idea is to take a matrix of the form $[I_n|A]$, where $A$ is an $n\times n$ matrix that has a special structure depending on the group rings used.

 Previously, group ring elements were used in a different way to construct certain special codes. In \cite{BLM}, an ideal of the group algebra $\FF_2S_4$ was used to construct the well-known binary extended Golay code where $S_4$ is the symmetric group on $4$ elements. In \cite{Hurley1}, an isomorphism between a group ring and a certain subring of the $n \times n$ matrices over the ring was established. This isomorphism was used to produce special self-dual codes in \cite{Hurley2,IT}.

The inspiration for this work comes from modifying the four-circulant construction, which was first introduced in \cite{BFC}: Let $G$ be the matrix

\begin{equation*}
\left[\ \ I_{2n}\ \ %
\begin{array}{|cc}
A & B \\
-B^{T} & A^{T}%
\end{array}%
\right]
\end{equation*}
where $A$ and $B$ are circulant matrices. Then the code generated by $G$ over $\F_p$ is self-dual if and only if $AA^T+BB^T = -I_n$. Note that when the alphabet is a ring of characteristic $2$, then the matrix and the conditions can be written in an alternative form, where the negative signs disappear.

In this work, we will consider constructing self dual codes from the following variation of the four-circulant matrix: Consider the matrix

\begin{equation*}
\left[\ \ I_{2n}\ \
\begin{array}{|cc}
A & B \\
B^{T} & A^{T}%
\end{array}%
\right]
\end{equation*}
where both $A$ and $B$  are matrices that arise from group rings. Depending on the groups, the matrices will usually not be circulant matrices, which is a variation from the usual four-circulant construction. Under this construction, we establish the link between units/non-units in the group ring and corresponding self-dual codes. Using this connection for some particular examples of groups over the field $\F_2$ and the rings $\F_2+u\F_2$ and $\F_4+u\F_4$ we are able to construct many extremal and optimal binary self-dual codes of different lengths. In particular we construct five new self-dual codes of parameters $[56, 28,10]$, twenty-three extremal binary self-dual codes of length 68 with new weight enumerators and fifteen new self-dual codes of parameters $[80,40,14]$.

The rest of the work is organized as follows. In section 2, we give the necessary background on codes, the alphabets we use and the group rings. In section 3, we give the constructions and the theoretical results about the group ring elements that lead to self-dual codes. In sections 4 and 5, we apply the construction methods to produce the numerical results, using MAGMA (\cite{MAGMA}). The paper ends with concluding remarks and possible further research directions.

\section{Preliminaries}

In this section, we will define self-dual codes over Frobenius rings of characteristic $2$. We will recall some of the properties of the family of rings called $R_k$ and the ring $\FF_4 +u \FF_4$. This section
concludes with an introduction to group rings and an established isomorphism between a group ring and a certain subring of the $n \times n$ matrices over a ring.

\subsection{Self-Dual codes}

Throughout this work, all rings are assumed to be commutative, finite, Frobenius rings with a multiplicative identity.

 A code over a finite commutative ring  $R$ is said to be any subset $C$ of $R^n$.  When the code is a submodule of the ambient space then the code is said to be linear.  To the ambient space, we attach the usual inner-product, specifically  $[\vv,\vw]= \sum v_i w_i$.
 The dual with respect to this inner-product is defined as $C^\perp = \{\vw \ | \ \vw \in R^n, \ [\vw,\vv]=0, \forall \vv \in C \}.$ Since the ring is Frobenius we have that for all linear codes over $R$, $|C| |C^\perp| = |R|^n$. If a code satisfies $C=C^\perp$ then the code $C$ is said to be self-dual.  If $C \subseteq C^\perp$ then the code is said to be self-orthogonal.

For binary codes, a self-dual code where all weights are divisible by 4, is said to be Type~II and the code is said to be Type~I otherwise. The bounds on the minimum distances for binary self-dual codes are give in the following theorem:
\begin{thm}
$($\cite{Rains}$)$ Let $d_{I}(n)$ and $d_{II}(n)$ be the minimum distance of
a Type I and Type II binary code of length $n$, respectively. Then
\begin{equation*}
d_{II}(n)\leq 4\lfloor \frac{n}{24}\rfloor +4
\end{equation*}%
and
\begin{equation*}
d_{I}(n)\leq \left\{
\begin{array}{ll}
4\lfloor \frac{n}{24}\rfloor +4 & \text{if $n\not\equiv 22\pmod{24}$} \\
4\lfloor \frac{n}{24}\rfloor +6 & \text{if $n\equiv 22\pmod{24}$.}%
\end{array}%
\right.
\end{equation*}
\end{thm}

Self-dual codes that meet these bounds are called \textit{extremal}. In sections 4 and 5, we will construct extremal binary self-dual codes of different lengths.

\subsection{$R_{k}$ family of rings}
An important class of rings that have been used extensively in constructing codes is the ring family of $R_k$, which has been introduced in \cite{Rkpaper}. We will be mainly using $R_0=\F_2$ and $R_1=\F_2+u\F_2$ in this work, however some of the theoretical results will be true for all $R_k$, which is why we would like to give a brief description of the rings, mainly from \cite{Rkpaper} and \cite{Rkpaper2}.  For $k\geq 1$, define
 \begin{equation} \label{defineRk}
R_k = \FF_2[u_1,u_2,\dots,u_k] / \langle u_i^2, u_iu_j- u_j u_i \rangle,
\end{equation}
which can also be defined recursively as:
 \begin{equation} \label{defineRk2}
R_k = R_{k-1}[u_k] / \langle u_k^2, u_ku_j- u_j u_k \rangle =
R_{k-1}+u_kR_{k-1}.
\end{equation}

For any subset $A \subseteq \{1,2, \dots, k\}$ we will fix
\begin{equation}
u_A := \prod_{i \in A}u_i
\end{equation}
with the convention that $u_{\emptyset} = 1.$
Then any element of $R_k$ can be represented as
\begin{equation} \label{eq4}
\sum_{A \subseteq \{1, \dots, k\}}c_Au_A, \qquad c_A \in \mathbb{F}_2.
\end{equation}

With this representation of the elements, we have
$$
u_A u_B = \left \{ \begin{array}{ll}
0 & \textrm{if} A \cap B \neq \emptyset\\
u_{A\cup B} & \textrm {if} A\cap B = \emptyset
\end{array} \right .
$$
and
$$\left ( \sum_{A}c_Au_A \right )\left ( \sum_{B}d_Bu_B\right) = \sum_{A,B \subseteq \{1,\dots,k\},  \\
A\cap B = \emptyset} c_Ad_Bu_{A \cup B}.$$

It is shown in \cite{Rkpaper} that the ring family $R_k$ is a commutative ring with $|R_k| = 2^{(2^k)}.$

A Gray map from $R_k$ to $\FF_2^{2^k}$ was defined inductively starting with the map on
$R_1$:
$\phi_1(a+bu_1) = (b,a+b).$
We recall that if $c \in R_k$, $c$ can be written as
$c= a + b u_{k-1}, a,b \in R_{k-1}.$
Then  \begin{equation} \phi_k(c) = (\phi_{k-1}(b),   \phi_{k-1}(a+b)).
\end{equation}

The map $\phi_k$ is a distance preserving map and the following is shown in \cite{Rkpaper2}.

\begin{thm}
Let $C$ be a self-dual code over $R_k$, then $\phi_k(R_k)$ is a binary self-dual code of length $2^kn.$
\end{thm}

The following lemma describes a property of units and non-units in $R_k$.
\begin{lem} $($ \cite{Rkpaper}$)$
\label{lemma:rk} For an element $\alpha\in R_k$ we have
\begin{equation}\label{unitsquare}
 \alpha^2 = \left \{
\begin{array}{ll}
1 & \textrm{if $\alpha$ is a unit}
\\
0 & \textrm {otherwise.}
\end{array}\right.
\end{equation}
\end{lem}

 The next result which was introduced in \cite{dougherty} can easily be extended to be true for $R_k$ as well.
\begin{thm}\label{extension}
Let $\mathcal{C}$ be a self-dual code
over $R_k$ of length $n$ and $G=(r_{i})$ be a $j\times n$ generator
matrix for $\mathcal{C}$, where $r_{i}$ is the $i$-th row of $G$, $1\leq
i\leq k$. Let $c$ be a unit in $R_k$  and $X$ be
a vector in ${R_k}^{n}$ with $\left\langle X,X\right\rangle =1$. Let $%
y_{i}=\left\langle r_{i},X\right\rangle $ for $1\leq i\leq k$. Then the
following matrix%
\begin{equation*}
\left(
\begin{array}{cc|c}
1 & 0 & X \\ \hline
y_{1} & cy_{1} & r_{1} \\
\vdots  & \vdots  & \vdots  \\
y_{k} & cy_{k} & r_{k}%
\end{array}%
\right) ,
\end{equation*}%
generates a self-dual code $\mathcal{C}^{\prime }$ over $R_k$ of
length $n+2$.
\end{thm}

\subsection{The ring $\FF_4+u\FF_4$}

Let $\mathbb{F}_{4}=\mathbb{F}_{2}\left( \omega \right) $ be the quadratic
field extension of $\mathbb{F}_2$, where $\omega ^{2}+\omega +1=0$. The ring
$\mathbb{F}_{4}+u\mathbb{F}_{4}$ is defined via $u^{2}=0$. Note that $%
\mathbb{F}_4+u\mathbb{F}_4$ can be viewed as an extension of $R_1=\mathbb{F}_2+u%
\mathbb{F}_2$ and so we can describe any element of $\mathbb{F}_4+u\mathbb{F}%
_4$ in the form $\omega a+\bar{\omega}b$ uniquely, where $a,b \in \mathbb{F}%
_2+u\mathbb{F}_2$.

A linear code $C$ of length $n$ over $\mathbb{F}_{4}+u\mathbb{F}_{4}$ is an $\left(
\mathbb{F}_{4}+u\mathbb{F}_{4}\right) $-submodule of $\left( \mathbb{F}_{4}+u%
\mathbb{F}_{4}\right) ^{n}$. In \cite{gaborit} and \cite{dougherty2} the following Gray maps were introduced;%
\begin{equation*}
\begin{tabular}{l||l}
$\psi _{\mathbb{F}_{4}}:\left( \mathbb{F}_{4}\right) ^{n}\rightarrow \left(
\mathbb{F}_{2}\right) ^{2n}$ & $\varphi _{\mathbb{F}_{2}+u\mathbb{F}%
_{2}}:\left( \mathbb{F}_{2}+u\mathbb{F}_{2}\right) ^{n}\rightarrow \mathbb{F}%
_{2}^{2n}$ \\
$a\omega +b\overline{\omega }\mapsto \left( a,b\right) \text{, \ }a,b\in
\mathbb{F}_{2}^{n}$ & $a+bu\mapsto \left( b,a+b\right) \text{, \ }a,b\in
\mathbb{F}_{2}^{n}.$%
\end{tabular}%
\end{equation*}
Note that $\varphi_{\mathbb{F}_{2}+u\mathbb{F}_{2}}$ is the same map as $\phi_1$ described before.
Those were generalized to the following maps in \cite{ling};
\begin{equation*}
\begin{tabular}{l||l}
$\psi _{\mathbb{F}_{4}+u\mathbb{F}_{4}}:\left( \mathbb{F}_{4}+u\mathbb{F}%
_{4}\right) ^{n}\rightarrow \left( \mathbb{F}_{2}+u\mathbb{F}_{2}\right)
^{2n}$ & $\varphi _{\mathbb{F}_{4}+u\mathbb{F}_{4}}:\left( \mathbb{F}_{4}+u%
\mathbb{F}_{4}\right) ^{n}\rightarrow \mathbb{F}_{4}^{2n}$ \\
$a\omega +b\overline{\omega }\mapsto \left( a,b\right) \text{, \ }a,b\in
\left( \mathbb{F}_{2}+u\mathbb{F}_{2}\right) ^{n}$ & $a+bu\mapsto \left(
b,a+b\right) \text{, \ }a,b\in \mathbb{F}_{4}^{n}$%
\end{tabular}%
\end{equation*}%
These maps preserve orthogonality in the corresponding alphabets. The binary
images $\varphi _{\mathbb{F}_{2}+u\mathbb{F}_{2}}\circ \psi _{\mathbb{F}%
_{4}+u\mathbb{F}_{4}}\left( C\right) $ and $\psi _{\mathbb{F}_{4}}\circ
\varphi _{\mathbb{F}_{4}+u\mathbb{F}_{4}}\left( C\right) $ are equivalent.
The Lee weight of an element is defined to be the Hamming weight of its
binary image.

\begin{prop}
$($\cite{ling}$)$ Let $C$ be a code over $\mathbb{F}_{4}+u\mathbb{F}_{4}$. If $C$ is self-orthogonal, so are $\psi _{\mathbb{F}_{4}+u\mathbb{F}%
_{4}}\left( C\right) $ and $\varphi _{\mathbb{F}_{4}+u\mathbb{F}_{4}}\left(
C\right) $. $C$ is a Type I (resp. Type II) code over $\mathbb{F}_{4}+u%
\mathbb{F}_{4}$ if and only if $\varphi _{\mathbb{F}_{4}+u\mathbb{F}%
_{4}}\left( C\right) $ is a Type I (resp. Type II) $\mathbb{F}_{4}$-code, if
and only if $\psi _{\mathbb{F}_{4}+u\mathbb{F}_{4}}\left( C\right) $ is a
Type I (resp. Type II) $\mathbb{F}_{2}+u\mathbb{F}_{2}$-code. Furthermore,
the minimum Lee weight of $C$ is the same as the minimum Lee weight of $\psi
_{\mathbb{F}_{4}+u\mathbb{F}_{4}}\left( C\right) $ and $\varphi _{\mathbb{F}%
_{4}+u\mathbb{F}_{4}}\left( C\right) $.
\end{prop}

\begin{cor}
Suppose that $C$ is a self-dual code over $\mathbb{F}_{4}+u\mathbb{F}_{4}$
of length $n$ and minimum Lee distance $d$. Then $\varphi _{\mathbb{F}_{2}+u%
\mathbb{F}_{2}}\circ \psi _{\mathbb{F}_{4}+u\mathbb{F}_{4}}\left( C\right) $
is a binary $\left[ 4n,2n,d\right] $ self-dual code. Moreover, $C$ and $%
\varphi _{\mathbb{F}_{2}+u\mathbb{F}_{2}}\circ \psi _{\mathbb{F}_{4}+u%
\mathbb{F}_{4}}\left( C\right) $ have the same weight enumerator. If $C$ is
Type I (Type II), then so is $\varphi _{\mathbb{F}_{2}+u\mathbb{F}_{2}}\circ
\psi _{\mathbb{F}_{4}+u\mathbb{F}_{4}}\left( C\right) $.
\end{cor}

\subsection{Shorthand notations for elements of $\F_2+u\F_2$ and $\F_4+u\F_4$}
In subsequent sections we will be writing tables in which vectors with elements from the rings $\F_2+u\F_2$ and $\F_4+u\F_4$ will appear.
In order to avoid writing long vectors with elements that can be confused with other elements, we will be describing the elements of these rings in a shorthand way, which will make the tables more compact.

For elements of $\F_2+u\F_2$, $0, 1, u$ will be written as they are, while $1+u$ will be replaced by 3. So, for example a vector of the form $(1,1+u,0,0,u,1+u)$ will be described as $(1300u3)$.

For the elements of $\FF_4+u\FF_4$, we use the ordered basis $\{u \omega, \omega, u, 1 \}$ to express the elements of $\FF_4+u\FF_4$ as binary strings of length 4. Then we will use the hexadecimal number system to describe each element:

$0 \leftrightarrow 0000$, $1 \leftrightarrow 0001$, $2 \leftrightarrow 0010$, $3 \leftrightarrow 0011$,
$4 \leftrightarrow 0100$, $5 \leftrightarrow 0101$, $6 \leftrightarrow 0110$, $7 \leftrightarrow 0111$,
$8 \leftrightarrow 1000$, $9 \leftrightarrow 1001$, $A \leftrightarrow 1010$, $B \leftrightarrow 1011$,
$C \leftrightarrow 1100$, $D \leftrightarrow 1101$, $E \leftrightarrow 1110$, $F \leftrightarrow 1111$.\\

For example $1 + u \omega$ corresponds to $1001$, which is represented by the hexadecimal $9$, while $\omega+u\omega$ corresponds to $1100$, which is represented by $C$.

\subsection{Certain Matrices and Group Rings}

We start with description of circulant matrices and block circulant matrices, the details for which can be found in \cite{davis}.

\begin{defi}
A circulant matrix over a ring $R$ is a square $n \times n$
matrix, which takes the form
\[ \mbox{circ}(a_1,a_2,\dots,a_n) = \left( \begin{array}{ccccc }

a_1 & a_2 & a_3 & \hdots & a_n\\

a_n & a_1 & a_2 & \hdots &  a_{n-1}\\

a_{n-1} & a_n & a_1 & \hdots & a_{n-2}\\

\vdots & \vdots & \vdots & \ddots &\vdots \\

a_2 & a_3 & a_4 & \hdots & a_1

\end{array}
\right)\] \noindent where $a_i \in R$.
\end{defi}

\begin{defi}
A block circulant matrix over a ring $R$ is a square $kn \times kn$
matrix, which takes the form
\[ \mbox{CIRC}(A_1,A_2,\dots,A_n) = \left( \begin{array}{ccccc }
A_1 & A_2 & A_3 & \hdots & A_n\\
A_n & A_1 & A_2 & \hdots &  A_{n-1}\\
A_{n-1} & A_n & A_1 & \hdots & A_{n-2}\\
\vdots & \vdots & \vdots & \ddots &\vdots \\
A_2 & A_3 & A_4 & \hdots & A_1
\end{array}
\right)\] \noindent where each $A_i$ is a $k \times k$ matrix over $R$.
\end{defi}

Let $G$ be a finite group or order $n$, then the group ring $RG$ consists of $\sum_{i=1}^n \alpha_i g_i$, $\alpha_i \in R$, $g_i \in G$. Addition in the group ring is done by coordinate addition, namely
\begin{equation} \sum_{i=1}^n \alpha_i g_i + \sum_{i=1}^n \beta_i g_i =
\sum_{i=1}^n (\alpha_i + \beta_i ) g_i.\end{equation}  The product of two elements in  a group ring  is given by
\begin{equation} \left(\sum_{i=1}^n \alpha_i g_i \right) \left( \sum_{j=1}^n \beta_j g_j \right)  = \sum_{i,j} \alpha_i \beta_j g_i g_j. \end{equation}  It follows  that the coefficient of $g_k$ in the product is $ \sum_{g_i g_j = g_k } \alpha_i \beta_j .$

The following construction of a matrix was first given by Hurley in \cite{Hurley1}. Let $R$ be a finite commutative Frobenius ring of characteristic $2$ and let $G = \{ g_1,g_2,\dots,g_n \}$ be a group of order $n$.  Let $v = \ds{\sum_{i=1}^n}\alpha_{g_i}g_i  \in RG.$  Define the matrix $\sigma(v) \in M_n(R)$ to be $\sigma(v)=(\alpha_{g_i^{-1} g_j})$ where $i,j \in  \{1, \ldots, n\}$. We note that the elements $g_1^{-1}, g_2^{-1}, \dots, g_n^{-1}$ are  the elements of the group $G$ in a given order.

Recall the canonical involution $%
*:RG\rightarrow RG$ on a group ring $RG$ is given by $v^* =
\sum_{g}a_gg^{-1} $, for $v = \sum_{g}a_g g \in RG$. An important connection
between $v^*$ and $v$ appears when we take their images under the $\sigma$
map:
\begin{equation}  \label{sigmavstar}
\sigma(v^*) = \sigma(v)^{T}.
\end{equation}

We will now
describe $\sigma(v)$ for the following group rings $RG$ where $G \in \{C_{n},C_{m} \times C_n,C_{m,n}\}$.
\begin{itemize}
\item Let $G=\langle x \,|\,x^{n}=1 \rangle \cong C_{n}$. If $v=\ds{ \sum_{i=0}^{n-1}} \alpha_{i+1}x^i\in RC_{n}$, then $\sigma(v)=\mbox{circ}(a_0,a_1,\dots,a_{n-1})$.
\item Let $G =\langle x,y \,|\,x^{n}=y^m=1,\,xy=yx \rangle \cong C_{m} \times C_n$ If $\ds{ v =\sum_{i=0}^{m-1} \sum_{j=0}^{n-1} a_{1+i+mj}x^{i}y^j} \in R(C_{m} \times C_n)$, then
\[ \sigma(v)=CIRC(A_1,\ldots,A_n)\]
where $A_{j+1}=circ(a_{1+mj},a_{2+mj},\ldots ,a_{m+mj})$, $a_{i}\in R$ and $m,n\geq 2$.
\item Let $G=C_{m,n}=\langle x \,|\,
x^{mn}=1 \rangle$. If $\ds{ v =\sum_{i=0}^{m-1} \sum_{j=0}^{n-1} a_{1+i+mj}x^{ni+j}} \in RC_{m,n}$, then

\[
\sigma (v )=%
\left(\begin{smallmatrix}
A_{1} & A_{2} & A_{3} & A_{4} & \cdots & A_{n-1} & A_{n} \\
A_{n}^{\prime } & A_{1} & A_{2} & A_{3} & \cdots & A_{n-2} & A_{n-1} \\
A_{n-1}^{\prime } & A_{n}^{\prime } & A_{1} & A_{2} & \cdots & A_{n-3} &
A_{n-2} \\
A_{n-2}^{\prime } & A_{n-1}^{\prime } & A_{n}^{\prime } & A_{1} & \cdots &
A_{n-4} & A_{n-4} \\
\vdots & \vdots & \vdots & \vdots & \ddots & \vdots & \vdots \\
A_{3}^{\prime } & A_{4}^{\prime } & A_{5}^{\prime } & A_{6}^{\prime } &
\cdots & A_{1} & A_{2} \\
A_{2}^{\prime } & A_{3}^{\prime } & A_{4}^{\prime } & A_{5}^{\prime } &
\cdots & A_{n}^{\prime } & A_{1}%
\end{smallmatrix} \right)
\]
where $A_{j+1}=circ(a_{1+mj},a_{2+mj},\ldots ,a_{m+mj})$, $A_{j+1}^{\prime }=circ(a_{m+mj},a_{1+mj},\dots ,a_{(m-1)+mj})$, $a_{i}\in R$ and $m,n\geq 2$. Note that $C_{m,n}$ is the same as the cyclic group $C_{mn}$, however we give a different labeling to the elements, which makes $\sigma(v)$ to be different than the matrix that we obtain from the standard labeling of cyclic groups.
\end{itemize}

\section{The Construction}

Let $v_1,v_2 \in RG $ where $R$ is a finite commutative Frobenius ring of characteristic $2$, $G$ is a finite group of order $p$ (where $p$ is odd) and $\gamma_i \in R$. Define the following matrix:

\begin{center}
\begin{tikzpicture}[scale=0.8]
\node at (0,0) {$\gamma_1$};\node at (0.75,0) {$\gamma_2$};\node at (1.5,0) {$\cdots$};\node at (2.25,0) {$\gamma_2$};
\node at (0,-0.75) {$\gamma_2$};\node at (0,-1.40) {$\vdots$};\node at (0,-2.25) {$\gamma_2$};
\draw (0.375,0.2) --(0.375,-2.4);\draw (-0.2,-0.375) --(2.4,-0.375);\node at (1.3875,-1.3875) {$\sigma(v_1)$};
\coordinate (G) at (-0.3,0.2);
\coordinate (R) at (-0.3,-2.4);
\draw [black]   (G) to[out=260,in=100] (R);
\coordinate (G1) at (2.5,0.2);
\coordinate (R1) at (2.5,-2.4);
\draw [black]   (G1) to[out=280,in=80] (R1);

\node at (3.75,0) {$\gamma_3$};\node at (4.5,0) {$\gamma_4$};\node at (5.25,0) {$\cdots$};\node at (6.0,0) {$\gamma_4$};
\node at (3.75,-0.75) {$\gamma_4$};\node at (3.75,-1.40) {$\vdots$};\node at (3.75,-2.25) {$\gamma_4$};
\coordinate (G) at (3.5,0.2);
\coordinate (R) at (3.5,-2.4);
\draw [black]   (G) to[out=260,in=100] (R);
\coordinate (G1) at (6.25,0.2);
\coordinate (R1) at (6.25,-2.4);
\draw [black]   (G1) to[out=280,in=80] (R1);
\draw (4.125,0.2) --(4.125,-2.4);\draw (3.55,-0.375) --(6.15,-0.375);\node at (5.25,-1.3875) {$\sigma(v_2)$};

\node at (0,-3) {$\gamma_3$};\node at (0.75,-3) {$\gamma_4$};\node at (1.5,-3) {$\cdots$};\node at (2.25,-3) {$\gamma_4$};
\node at (0,-3.75) {$\gamma_4$};\node at (0,-4.40) {$\vdots$};\node at (0,-5.25) {$\gamma_4$};
\draw (0.375,-2.8) --(0.375,-5.4);\draw (-0.2,-3.375) --(2.4,-3.375);\node at (1.3875,-4.3875) {$\sigma(v_2)$};
\coordinate (G) at (-0.3,-2.8);
\coordinate (R) at (-0.3,-5.4);
\draw [black]   (G) to[out=260,in=100] (R);
\coordinate (G1) at (2.5,-2.8);
\coordinate (R1) at (2.5,-5.4);
\draw [black]   (G1) to[out=280,in=80] (R1);
\node [scale=0.75] at (2.7,-2.9) {$T$};

\node at (3.75,-3) {$\gamma_1$};\node at (4.5,-3) {$\gamma_2$};\node at (5.25,-3) {$\cdots$};\node at (6.0,-3) {$\gamma_2$};
\node at (3.75,-3.75) {$\gamma_2$};\node at (3.75,-4.40) {$\vdots$};\node at (3.75,-5.25) {$\gamma_2$};
\coordinate (G) at (3.5,-2.8);
\coordinate (R) at (3.5,-5.4);
\draw [black]   (G) to[out=260,in=100] (R);
\coordinate (G1) at (6.25,-2.8);
\coordinate (R1) at (6.25,-5.4);
\draw [black]   (G1) to[out=280,in=80] (R1);
\draw (4.125,-2.8) --(4.125,-5.4);\draw (3.55,-3.375) --(6.15,-3.375);\node at (5.25,-4.3875) {$\sigma(v_1)$};
\node [scale=0.75] at (6.45,-2.9) {$T$};
\draw  [line width=0.35mm] (6.7,0.2)--(6.8,0.2) --(6.8,-5.4)--(6.7,-5.4);
\draw  [line width=0.35mm] (-0.8,0.2) --(-0.8,-5.4);
\draw  [line width=0.35mm] (3.0,0.2) --(3.0,-5.4);
\draw  [line width=0.35mm] (-4.7,0.2)--(-4.8,0.2) --(-4.8,-5.4)--(-4.7,-5.4);
\draw  [line width=0.35mm] (-0.8,-2.6) --(6.8,-2.6);
\node [scale=1] at (-2.8,-2.6) {$I_{2p+2}$};
\node [scale=1] at (-5.6,-2.6) {$M_{\sigma}=$};
\draw  [line width=0.35mm] (-4.7,-5.8)--(-4.8,-5.8) --(-4.8,-10.2)--(-4.7,-10.2);
\node [scale=1] at (-5.2,-8) {$=$};
\node [scale=1] at (-2.8,-8) {$I_{2p+2}$};
\draw  [line width=0.35mm] (-0.8,-5.8) --(-0.8,-10.2);
\draw  [line width=0.35mm] (1.9,-5.8) --(1.9,-10.2);
\draw  [line width=0.35mm] (4.5,-5.8)--(4.6,-5.8) --(4.6,-10.2) --(4.5,-10.2);
\draw  [line width=0.35mm] (-0.8,-8) --(4.6,-8);
\node at (-0.4,-6) {$\gamma_1$};\node at (0.3,-6) {$\gamma_2$};\node at (0.95,-6) {$\cdots$};\node at (1.5,-6) {$\gamma_2$};
\node at (-0.4,-6.5) {$\gamma_2$};\node at (-0.4,-7.0) {$\vdots$};\node at (-0.4,-7.7) {$\gamma_2$};

\node at (2.3,-6) {$\gamma_3$};\node at (3,-6) {$\gamma_4$};\node at (3.65,-6) {$\cdots$};\node at (4.2,-6) {$\gamma_4$};
\node at (2.3,-6.5) {$\gamma_4$};\node at (2.3,-7.0) {$\vdots$};\node at (2.3,-7.7) {$\gamma_4$};

\node at (-0.4,-8.3) {$\gamma_3$};\node at (0.3,-8.3) {$\gamma_4$};\node at (0.95,-8.3) {$\cdots$};\node at (1.5,-8.3) {$\gamma_4$};
\node at (-0.4,-8.8) {$\gamma_4$};\node at (-0.4,-9.3) {$\vdots$};\node at (-0.4,-9.95) {$\gamma_4$};

\node at (2.3,-8.3) {$\gamma_1$};\node at (3,-8.3) {$\gamma_2$};\node at (3.65,-8.3) {$\cdots$};\node at (4.2,-8.3) {$\gamma_2$};
\node at (2.3,-8.8) {$\gamma_2$};\node at (2.3,-9.3) {$\vdots$};\node at (2.3,-9.95) {$\gamma_2$};
\draw  (-0.05,-5.8) --(-0.05,-10.2);
\draw  (2.65,-5.8) --(2.65,-10.2);
\draw  (-0.8,-8.55) --(4.6,-8.55);
\draw  (-0.8,-6.25) --(4.6,-6.25);
\node at (0.87,-7.125) {$\sigma(v_1)$};
\node at (0.97,-9.325) {$\sigma(v_2)^T$};
\node at (3.57,-7.125) {$\sigma(v_2)$};
\node at (3.67,-9.325) {$\sigma(v_1)^T$};
\end{tikzpicture}
\end{center}

\noindent Let $C_{\sigma}$ be a code that is generated by the matrix $M(\sigma)$. Then, the code $C_{\sigma}$ has length $4p+4$. We aim to establish some restrictions when
this construction yields self-dual codes.

\begin{thm}
Let $R$ be a finite commutative Frobenius ring of characteristic $2$ and let $G=\{g_1,g_2,\ldots,g_p\}$ be a finite group of order $p$ (where $p$ is odd). If
\begin{enumerate}
\item $v_1v_2=v_2v_1$
\item $\sum_{i=1}^4\gamma_i^2=1$,
\item $v_1v_1^*+v_2v_2^*+(\gamma_2+\gamma_4)^2\widehat{g}+1=0$,
\item $v_1^*v_1+v_2^*v_2+(\gamma_2+\gamma_4)^2\widehat{g}+1=0$ and
\item $\gamma_1=\delta_1$ and $\gamma_3=\delta_2$
\end{enumerate}
then $C_{\sigma}$ is a self-dual code of length $4p+4$ where $\widehat{g}=\sum_{i=1}^pg_i$, $v_1=\sum_{g \in G}\alpha_g g$, $v_2=\sum_{g \in G}\beta_g g$, $\delta_1=\sum_{g \in G}\alpha_g$ and  $\delta_2=\sum_{g \in G}\beta_g $.
\end{thm}
\begin{proof}Clearly, $C_{\sigma}$ has free rank $2p+2$ as the left hand side of the generator matrix is the $4p+4$ by $4p+4$ identity matrix. Let us consider $M(\sigma)M(\sigma)^T$. Let
\resizebox{0.3\hsize}{!}{$M(\sigma)=\left(\ \ I_{2n}\ \
\begin{array}{|cc}
A & B \\
B^{T} & A^{T}%
\end{array}%
\right)$}, then
\newline $M(\sigma)M(\sigma)^T=\left(\begin{smallmatrix}AA^T+BB^T+I & AB+BA\\ B^TA^T+A^TB^T& B^TB+A^TA+I \end{smallmatrix} \right)$ where
\resizebox{0.2\hsize}{!}{$A=\left(\begin{array}{c|ccc}\gamma_1 & \gamma_2 & \cdots & \gamma_2\\ \hline \gamma_2 & & & \\ \vdots & &\sigma(v_1) & \\
\gamma_2 &  & & \end{array} \right)$} and \resizebox{0.2\hsize}{!}{$B=\left(\begin{array}{c|ccc}\gamma_3 & \gamma_4 & \cdots & \gamma_4 \\ \hline \gamma_4 & & & \\ \vdots & &\sigma(v_2) & \\
\gamma_4 &  & & \end{array} \right)$}.

\noindent Let $A=\left(\begin{smallmatrix}B_1&B_2\\B_2^T&B_3 \end{smallmatrix} \right)$ and $B=\left(\begin{smallmatrix}B_4&B_5\\B_5^T&B_6 \end{smallmatrix} \right)$ where $B_1=\gamma_1$,
$B_2=\left(\begin{smallmatrix}\gamma_2& \cdots &\gamma_2 \end{smallmatrix} \right)$, $B_3=\sigma(v_1)$, $B_4=\gamma_3$,
$B_5=\left(\begin{smallmatrix}\gamma_4& \cdots &\gamma_4 \end{smallmatrix} \right)$ and $B_6=\sigma(v_2)$. Now,
\[
AB+BA=\begin{pmatrix}
B_1B_4+B_2B_5^T+B_4B_1+B_5B_2^T & B_1B_5+B_2B_6+B_4B_2+B_5B_3\\
B_2^TB_4+B_3B_5^T+B_5^TB_1+B_6B_2^T& B_2^TB_5+B_3B_6+B_5^TB_2+B_6B_3
\end{pmatrix},
\]

\[
\begin{split}
B_1B_4+B_2B_5^T+B_4B_1+B_5B_2^T&=\gamma_1\gamma_3+\left(\begin{smallmatrix}\gamma_2& \cdots &\gamma_2 \end{smallmatrix} \right) \left(\begin{smallmatrix}\gamma_4\\ \vdots \\ \gamma_4 \end{smallmatrix} \right)+\gamma_3\gamma_1+\left(\begin{smallmatrix}\gamma_4& \cdots &\gamma_4 \end{smallmatrix} \right)\left(\begin{smallmatrix}\gamma_2\\ \vdots \\ \gamma_2 \end{smallmatrix} \right)\\
&= \gamma_1\gamma_3+ \gamma_2 \gamma_4+\gamma_3\gamma_1+  \gamma_4 \gamma_2=0,
\end{split}
\]

\[
\begin{split}
B_1B_5+B_2B_6+B_4B_2+B_5B_3
&=\gamma_1\left(\begin{smallmatrix}\gamma_4& \cdots &\gamma_4 \end{smallmatrix} \right)+
\left(\begin{smallmatrix}\gamma_2& \cdots &\gamma_2 \end{smallmatrix} \right)\sigma(v_2)+
\gamma_3\left(\begin{smallmatrix}\gamma_2& \cdots &\gamma_2 \end{smallmatrix} \right)+
\left(\begin{smallmatrix}\gamma_4& \cdots &\gamma_4 \end{smallmatrix} \right)\sigma(v_1)\\
&=\gamma_1\gamma_4\left(\begin{smallmatrix}1& \cdots &1 \end{smallmatrix} \right)+
\gamma_2\delta_2\left(\begin{smallmatrix}1& \cdots &1 \end{smallmatrix} \right)+
\gamma_3\gamma_2\left(\begin{smallmatrix}1& \cdots &1 \end{smallmatrix} \right)+
\gamma_4\delta_1\left(\begin{smallmatrix}1& \cdots &1 \end{smallmatrix} \right),\\
&=(\gamma_1\gamma_4+
\gamma_2\delta_2+
\gamma_3\gamma_2+
\gamma_4\delta_1)\left(\begin{smallmatrix}1& \cdots &1 \end{smallmatrix} \right)=0,
\end{split}
\]

\[
\begin{split}
B_2^TB_4+B_3B_5+B_5^TB_1+B_6B_2^T
&=\left(\begin{smallmatrix}\gamma_2\\ \vdots \\ \gamma_2 \end{smallmatrix} \right)\gamma_3+
\sigma(v_1)\left(\begin{smallmatrix}\gamma_4\\ \vdots \\ \gamma_4 \end{smallmatrix} \right)+
\left(\begin{smallmatrix}\gamma_4\\ \vdots \\ \gamma_4 \end{smallmatrix} \right)\gamma_1+
\sigma(v_2)\left(\begin{smallmatrix}\gamma_2\\ \vdots \\ \gamma_2 \end{smallmatrix} \right)\\
&=\gamma_2\gamma_3\left(\begin{smallmatrix}1\\ \vdots \\ 1 \end{smallmatrix} \right)+
\delta_1\gamma_4\left(\begin{smallmatrix}1\\ \vdots \\ 1\end{smallmatrix} \right)+
\gamma_4\gamma_1\left(\begin{smallmatrix}1\\ \vdots \\ 1 \end{smallmatrix} \right)+
\delta_2\gamma_2\left(\begin{smallmatrix}1\\ \vdots \\ 1\end{smallmatrix} \right)\\
&=(\gamma_2\gamma_3+\delta_1\gamma_4+\gamma_4\gamma_1+\delta_2\gamma_2)\left(\begin{smallmatrix}1\\ \vdots \\ 1 \end{smallmatrix} \right)=0,
\end{split}
\]

\[
\begin{split}
B_2^TB_5+B_3B_6+B_5^TB_2+B_6B_3 &=\left(\begin{smallmatrix}\gamma_2\\ \vdots \\ \gamma_2 \end{smallmatrix} \right)\left(\begin{smallmatrix}\gamma_4& \cdots &\gamma_4 \end{smallmatrix} \right)+\sigma(v_1)\sigma(v_2)+\left(\begin{smallmatrix}\gamma_4\\ \vdots \\ \gamma_4 \end{smallmatrix} \right)\left(\begin{smallmatrix}\gamma_2& \cdots &\gamma_2 \end{smallmatrix} \right)+\sigma(v_2)\sigma(v_1)\\
&=\gamma_2 \gamma_4 \begin{pmatrix}1&\cdots&1\\ \vdots&\ddots&\vdots\\ 1&\cdots&1  \end{pmatrix}+\sigma(v_1v_2)+\gamma_4 \gamma_2 \begin{pmatrix}1&\cdots&1\\ \vdots&\ddots&\vdots\\ 1&\cdots&1  \end{pmatrix}+\sigma(v_2v_1)\\
&=\sigma(v_1v_2)+\sigma(v_2v_1)=0.
\end{split}
\]

Also,
\[ AA^T+BB^T+I=\begin{pmatrix}B_1^2+B_2B_2^T+B_4^2+B_5B_5^T+1& B_1B_2+B_2B_3^T+B_4B_5+B_5B_6^T\\
B_2^TB_1+B_3B_2^T+B_5^TB_4+B_6B_5^T & B_2^TB_2+B_3B_3^T+B_5^TB_5+B_6B_6^T +I \end{pmatrix},\]
\[
\begin{split}
B_1^2+B_2B_2^T+B_4^2+B_5B_5^T+1&=\gamma_1^2+\left(\begin{smallmatrix}\gamma_2& \cdots &\gamma_2 \end{smallmatrix} \right) \left(\begin{smallmatrix}\gamma_2\\ \vdots \\ \gamma_2 \end{smallmatrix} \right)+\gamma_3^2+\left(\begin{smallmatrix}\gamma_4& \cdots &\gamma_4 \end{smallmatrix} \right) \left(\begin{smallmatrix}\gamma_4\\ \vdots \\ \gamma_4 \end{smallmatrix} \right)=0\\
&=
(\gamma_1+\gamma_3)^2+p(\gamma_2+\gamma_4)^2+1\\
&=\sum_{i=1}^4\gamma_i^2+1,
\end{split}\]

\[
\begin{split}
B_1B_2+B_2B_3^T+B_4B_5+B_5B_6^T&=\gamma_1\left(\begin{smallmatrix}\gamma_2& \cdots &\gamma_2 \end{smallmatrix} \right)+ \left(\begin{smallmatrix}\gamma_2& \cdots &\gamma_2 \end{smallmatrix} \right) \sigma(v_1)^T+\gamma_3  \left(\begin{smallmatrix}\gamma_4& \cdots &\gamma_4 \end{smallmatrix} \right)+\left(\begin{smallmatrix}\gamma_4& \cdots &\gamma_4 \end{smallmatrix} \right)\sigma(v_2)^T\\
&=\left(\begin{smallmatrix}\gamma_1 \gamma_2& \cdots & \gamma_1\gamma_2 \end{smallmatrix} \right)+\left(\begin{smallmatrix}\gamma_2 \delta_1& \cdots &\gamma_2 \delta_1 \end{smallmatrix} \right)+  \left(\begin{smallmatrix} \gamma_3 \gamma_4& \cdots &\gamma_3 \gamma_4 \end{smallmatrix} \right)+\left(\begin{smallmatrix}\gamma_4 \delta_2& \cdots & \delta_2\gamma_4 \end{smallmatrix} \right)\\
&=(\gamma_1 \gamma_2+\gamma_2 \delta_1+\gamma_3 \gamma_4+\gamma_4 \delta_2)
\left(\begin{smallmatrix} 1 & \cdots & 1\end{smallmatrix} \right)=0,
\end{split}
\]

\noindent where $v_1=\sum_{g \in G}\alpha_g g$, $v_2=\sum_{g \in G}\beta_g g$, $\delta_1=\sum_{g \in G}\alpha_g$ and  $\delta_2=\sum_{g \in G}\beta_g $ and
\[
\begin{split}
B_2^TB_2+B_3B_3^T+B_5^TB_5+B_6B_6^T +I&=\left(\begin{smallmatrix}\gamma_2\\ \vdots \\ \gamma_2 \end{smallmatrix} \right)\left(\begin{smallmatrix}\gamma_2& \cdots &\gamma_2 \end{smallmatrix} \right)
+\sigma(v_1)\sigma(v_1)^T+\left(\begin{smallmatrix}\gamma_4\\ \vdots \\ \gamma_4 \end{smallmatrix} \right)\left(\begin{smallmatrix}\gamma_4& \cdots &\gamma_4 \end{smallmatrix} \right)
+\sigma(v_2)\sigma(v_2)^T+I\\
&=\sigma(v_1v_1^*)+\sigma(v_2v_2^*)+(\gamma_2+\gamma_4)^2\begin{pmatrix}1&\cdots&1\\ \vdots&\ddots&\vdots\\ 1&\cdots&1  \end{pmatrix}+I.
\end{split}
\]
\noindent Additionally,
\[ B^TB+A^TA+I=\begin{pmatrix}B_1^2+B_2B_2^T+B_4^2+B_5B_5^T+1& B_1B_2+B_2B_3+B_4B_5+B_5B_6\\
B_2^TB_1+B_3^TB_2^T+B_5^TB_4+B_6^TB_5^T & B_2^TB_2+B_3^TB_3+B_5^TB_5+B_6^TB_6 +I \end{pmatrix}\]

\[
\begin{split}
B_1B_2+B_2B_3+B_4B_5+B_5B_6&=
\gamma_1\left(\begin{smallmatrix}\gamma_2& \cdots &\gamma_2 \end{smallmatrix} \right)+ \left(\begin{smallmatrix}\gamma_2& \cdots &\gamma_2 \end{smallmatrix} \right) \sigma(v_1)+\gamma_3  \left(\begin{smallmatrix}\gamma_4& \cdots &\gamma_4 \end{smallmatrix} \right)+\left(\begin{smallmatrix}\gamma_4& \cdots &\gamma_4 \end{smallmatrix} \right)\sigma(v_2)\\
&=\left(\begin{smallmatrix}\gamma_1 \gamma_2& \cdots & \gamma_1\gamma_2 \end{smallmatrix} \right)+\left(\begin{smallmatrix}\gamma_2 \delta_1& \cdots &\gamma_2 \delta_1 \end{smallmatrix} \right)+  \left(\begin{smallmatrix} \gamma_3 \gamma_4& \cdots &\gamma_3 \gamma_4 \end{smallmatrix} \right)+\left(\begin{smallmatrix}\gamma_4 \delta_2& \cdots & \delta_2\gamma_4 \end{smallmatrix} \right)\\
&=(\gamma_1 \gamma_2+\gamma_2 \delta_1+\gamma_3 \gamma_4+\gamma_4 \delta_2)
\left(\begin{smallmatrix} 1 & \cdots & 1\end{smallmatrix} \right)=0,
\end{split}
\]

\noindent and
\[ B_2^TB_2+B_3^TB_3+B_5^TB_5+B_6^TB_6 +I=\sigma(v_1^*v_1)+\sigma(v_2^*v_2)+(\gamma_2+\gamma_4)^2\begin{pmatrix}1&\cdots&1\\ \vdots&\ddots&\vdots\\ 1&\cdots&1  \end{pmatrix}+I.\]
Finally,  $M(\sigma)M(\sigma)^T$ is a symmetric matrix and $C_{\sigma}$ is self-orthogonal when $\sigma(v_1v_2)=\sigma(v_2v_1)$, $\sum_{i=1}^4\gamma_i^2=1$,  $\gamma_1=\delta_1$, $\gamma_3=\delta_2$,
$v_1v_1^*+v_2v_2^*+(\gamma_2+\gamma_4)^2\widehat{g}+1=0$ and $v_1^*v_1+v_2^*v_2+(\gamma_2+\gamma_4)^2\widehat{g}+1=0$.
\end{proof}

\begin{cor}
Let $R=R_k$, $G$ be a finite group of order $p$ (where $p$ is odd), $\gamma_2 +\gamma_4$ be a non-unit in $R_k$ and $C_{\sigma}$ be self-dual. Then either
\begin{itemize}
\item $v_1 \in RG$ is a unitary unit and $v_2$ is a non-unit or
\item $v_2 \in RG$ is a unitary unit and $v_1$ is a non-unit.
\end{itemize}
\end{cor}
\begin{proof} Let $v_1 \in RG$ be a unitary unit and $\gamma_2 +\gamma_4$ be a non-unit in $R_k$. Then, $v_1v_1^*=1$ and $(\gamma_2 +\gamma_4)^2=0$.
Clearly $v_2v_2^*=0$, $det(\sigma(v_2v_2^*))=0$ is a non-unit by Corollary $3$ in \cite{Hurley1}. Therefore $v_2$ is a non-unit in $RG$. Similarly, if $v_2 \in RG$ is unitary
and $\gamma_2+\gamma_4$ is a non-unit in $R_k$, then $v_1$ is a non-unit in $RG$.

\end{proof}
\begin{cor}
Let $R=R_k$, $G$ be a finite group of order $p$ (where $p$ is odd), $\gamma_2+\gamma_4$ be a non-unit in $R_k$ and $C_{\sigma}$ be self-dual. Then $v_1^*v_1+v_2^*v_2$ is a non-unit in $RG$.
\end{cor}
\begin{proof}
Let $\gamma_2+\gamma_4$, then $(\gamma_2+\gamma_4)^2=1$. Now
\[
\begin{split}
\sigma(v_1^*v_1)+\sigma(v_2^*v_2)+(\gamma_2+\gamma_4)^2\begin{pmatrix}1&\cdots&1\\ \vdots&\ddots&\vdots\\ 1&\cdots&1  \end{pmatrix}+I&=0\\
\sigma(v_1^*v_1)+\sigma(v_2^*v_2)+\begin{pmatrix}1&\cdots&1\\ \vdots&\ddots&\vdots\\ 1&\cdots&1  \end{pmatrix}+I&=0\\
\sigma(v_1^*v_1+v_2^*v_2)&=\begin{pmatrix} 0&1&1& \cdots &1\\1&0&1& \cdots &1 \\1 &1 & 0 & \cdots &1\\ \vdots & \vdots & \vdots & \ddots & 1\\ 1&1&1& \cdots &0  \end{pmatrix}.
\end{split}
\]
\noindent Now,
\[
det \begin{pmatrix} 0&1&1& \cdots &1\\1&0&1& \cdots &1 \\1 &1 & 0 & \cdots &1\\ \vdots & \vdots & \vdots & \ddots & 1\\ 1&1&1& \cdots &0  \end{pmatrix}
=det \begin{pmatrix} p-1&p-1&p-1& \cdots &p-1\\1&0&1& \cdots &1 \\1 &1 & 0 & \cdots &1\\ \vdots & \vdots & \vdots & \ddots & 1\\ 1&1&1& \cdots &0  \end{pmatrix}
=(p-1)det \begin{pmatrix} 1&1&1& \cdots &1\\1&0&1& \cdots &1 \\1 &1 & 0 & \cdots &1\\ \vdots & \vdots & \vdots & \ddots & 1\\ 1&1&1& \cdots &0  \end{pmatrix}=0
\]
\noindent since $p$ is odd. Therefore $\sigma(v_1^*v_1+v_2^*v_2)=0$ and $v_1^*v_1+v_2^*v_2$ is non-unit by Corollary $3$ in \cite{Hurley1}.
\end{proof}
\section{Extremal binary self-dual codes from the constructions}

In this section, we will present the results obtained  using the construction described in section 3, to construct self-dual codes for certain groups over different alphabets.

\subsection{Constructions coming from $C_3$}
We apply the constructions over the rings $\F_4$ and $\F_4+u\F_4$. The Gray images of the self-dual codes are binary self-dual codes of length 32 and 64 respectively. We only list the extremal ones.
\begin{table}[H]
\caption{Extremal binary self-dual codes of length 32 from self-dual codes over $\FF_4$ of length $16$ via $C_{3}$}
\begin{center}
\scalebox{0.8}{
\begin{tabular}{cccccc}
\hline
$(\gamma_1,\gamma_2)$& $v_1$ & $(\gamma_3,\gamma_4)$&$v_2$ & $|Aut(\mathcal{C})|$ & Type\\ \hline \hline
$(0,1)$& $(0,1,1)$ & $(\w,\w)$&$(0,1,\w+1)$ & $2^9 \cdot 3^3 \cdot 5$ & $[32,16,8]_{II}$\\ \hline
$(0,1)$& $(1,\w,\w+1)$ & $(\w,\w)$&$(\w,\w+1,\w+1)$ & $2^{15} \cdot 3^2 \cdot 5 \cdot 7$ & $[32,16,8]_{II}$\\ \hline
$(\w,\w)$& $(0,0,\w)$ & $(\w,\w+1)$&$(1,1,\w)$ & $2^{15} \cdot 3^2$ & $[32,16,8]_{II}$\\ \hline
\end{tabular}}%
\end{center}
\end{table}

\bigskip

We recall that the possible weight enumerators for a self-dual Type I $\left[ 64,32,12\right]$-code is given in \cite{conway,binary} as:
\begin{eqnarray*}
W_{64,1} &=&1+\left( 1312+16\beta \right) y^{12}+\left( 22016-64\beta
\right) y^{14}+\cdots ,14\leq \beta \leq 284, \\
W_{64,2} &=&1+\left( 1312+16\beta \right) y^{12}+\left( 23040-64\beta
\right) y^{14}+\cdots ,0\leq \beta \leq 277.
\end{eqnarray*}%

With the most updated information, the existence of codes is known for $\beta =$14, 18, 22, 25,\ 29, 32, 35,
36, 39, 44, 46, 53,\ 59, 60, 64 and 74 in $W_{64,1}$ and for $\beta =$0, 1,
2, 4, 5,\ 6, 8, 9, 10, 12, 13,\ 14, 16$,\ldots ,$\ 25,\ 28,\ 19,\ 30, 32,\
33,\ 34, 36, 37, 38, 40,\ 41,\ 42, 44, 45, 48, 50, 51,\ 52,\ 56, 58, 64, 72,
80,\ 88,\ 96, 104, 108,\ 112,\ 114,\ 118,\ 120 and 184 in $W_{64,2}$.

\bigskip
\begin{table}[H]
\caption{Extremal binary self-dual codes of length 64 from self-dual codes over $\FF_4+u\FF_4$ of length 16 via $C_{3}$.}\label{T:1}
\begin{center}
\scalebox{0.8}{
\begin{tabular}{ccccccc}
\hline
$\mathcal{C}_i$ & $(\gamma_1,\gamma_2)$& $v_1$ & $(\gamma_3,\gamma_4)$&$v_2$ & $|Aut(\mathcal{C}_i)|$ & $W_{64,2}$\\ \hline \hline
1 & $(1,8)$& $(2,A,9)$ & $(6,6)$ & $(0,9,F)$ & $2^2 \cdot 3$ & $\beta=13$\\ \hline
2 &$(0,A)$&$(2,9,B)$&$(6,5)$&$(8,B,5)$ &$2^3 \cdot 3$ & $\beta=13$\\ \hline
3 &$(0,A)$&$(A,2,9)$&$(4,7)$&$(9,6,1)$ &$2^4 \cdot 3$ & $\beta=16$\\ \hline
4 &$(0,A)$&$(A,1,6)$&$(6,5)$&$(4,D,F)$ &$2^2 \cdot 3$ & $\beta=19$\\ \hline
5 &$(1,8)$&$(B,4,E)$&$(4,4)$&$(0,2,6)$ &$2^2 \cdot 3$ & $\beta=22$\\ \hline
6 &$(9,2)$&$(2,A,1)$&$(6,6)$&$(8,3,D)$ &$2^2 \cdot 3$ & $\beta=25$\\ \hline
7 &$(1,8)$&$(A,A,1)$&$(4,4)$&$(8,B,7)$ &$2^3 \cdot 3$ & $\beta=25$\\ \hline
8 &$(1,8)$&$(A,6,D)$&$(4,4)$&$(E,5,F)$ &$2^2 \cdot 3$ & $\beta = 37$\\ \hline
9 &$(2,9)$&$(1,E,D)$&$(4,E)$&$(4,F,F)$ &$2^3 \cdot 3$ & $\beta =37$\\ \hline
10 &$(0,A)$&$(0,9,9)$&$(4,7)$&$(0,1,5)$ &$2^4 \cdot 3^2$ & $\beta=40$\\ \hline
11 &$(0,A)$&$(0,9,9)$&$(4,7)$&$(2,9,F)$ &$2^4 \cdot 3$ & $\beta=64$\\ \hline
\end{tabular}}%
\end{center}
\end{table}

\subsection{Constructions coming from $C_7$}

We apply the constructions coming from $C_7$ over the binary field and the ring $R_1 =\F_2+u\F_2$ as a result of which we obtain extremal binary self-dual codes of lengths 32 and 64 respectively.
\begin{table}[H]
\caption{Extremal binary self-dual codes of length $32$ from $C_{7}$.}
\begin{center}
\scalebox{0.9}{
\begin{tabular}{cccccc}
\hline
$(\gamma_1,\gamma_2)$& $v_1$ & $(\gamma_3,\gamma_4)$&$v_2$ & $|Aut(\mathcal{C})|$ & Type \\ \hline \hline
$(0,0)$& $(0,0,0,0,0,1,1)$ & $(0,1)$&$(0,1,1,0,0,1,1)$ & $2^{12} \cdot 3 \cdot 7$ & $[32,16,8]_{II}$ \\ \hline
$(0,0)$& $(0,0,1,0,1,1,1)$ & $(0,1)$&$(0,1,1,1,1,1,1)$ & $2^{15} \cdot 3^2 \cdot 5 \cdot 7$ & $[32,16,8]_{II}$ \\ \hline
$(1,0)$& $(0,0,0,0,1,1,1)$ & $(1,1)$&$(1,1,0,1,0,1,1)$ & $2^{12} \cdot 3 \cdot 7$ & $[32,16,8]_{I}$ \\ \hline
\end{tabular}}%
\end{center}
\end{table}

\begin{table}[H]
\caption{Extremal binary self-dual codes of length 64 from self-dual codes over $\FF_2+u\FF_2$ of length $32$ via $C_{7}$.}
\begin{center}
\scalebox{0.9}{
\begin{tabular}{cccccc}
\hline
 $(\gamma_1,\gamma_2)$& $v_1$ & $(\gamma_3,\gamma_4)$&$v_2$ & $|Aut(\mathcal{C})|$ & $W_{64,2}$\\ \hline \hline
$(u,u)$ & $(u,0,0,u,0,1,3)$ & $(1,1)$ & $(u,1,1,0,u,3,1)$ & $2^2 \cdot 7$ & $\beta=16$ \\ \hline
$(u,u)$ & $(u,u,0,0,0,1,3)$ & $(1,1)$ & $(u,1,1,u,0,3,1)$ & $2^2 \cdot 7$ & $\beta=30$ \\ \hline
$(u,u)$ & $(u,0,1,0,1,1,1)$ & $(u,1)$ & $(u,1,1,3,1,1,3)$ & $2^2 \cdot 7$ & $\beta=37$ \\ \hline
$(u,u)$ & $(u,u,1,u,1,1,1)$ & $(u,1)$ & $(u,1,1,3,3,1,1)$ & $2^3 \cdot 3 \cdot 7 $ & $\beta=37$ \\ \hline
$(u,u)$ & $(u,0,u,0,0,1,3)$ & $(1,1)$ & $(u,1,1,u,0,1,3)$ & $2^2 \cdot 7$ & $\beta =44$ \\ \hline
$(u,u)$ & $(u,u,0,0,u,1,1)$ & $(u,1)$ & $(u,1,3,u,u,1,3)$ & $2^3 \cdot 7$ & $\beta=44$ \\ \hline
$(u,u)$ & $(u,0,1,0,1,3,3)$ & $(1,1)$ & $(u,1,1,1,3,3,1)$ & $2^2 \cdot 7$ & $\beta=51$ \\ \hline
$(u,u)$ & $(u,u,u,u,u,1,1)$ & $(u,1)$ & $(u,1,3,u,u,,1)$ & $2^4 \cdot 3 \cdot 7$ & $\beta=72$ \\ \hline
\end{tabular}}%
\end{center}
\end{table}

\subsection{Constructions coming from groups of order $9$}
We apply the constructions $C_9, C_{3,3}$ and $C_{3}\times C_3$ over the binary field and the ring $R_1=\F_2+u\F_2$, as a result of which we obtain binary self-dual codes of lengths 40 and $80$. For the length 40, we get extremal self-dual codes and for length 80, we get the best Type I codes, i.e., self-dual codes that have parameters $[80,40,14]$.

\begin{table}[H]
\caption{Extremal binary self-dual codes of length $40$ from $C_9$, $C_{3,3}$ and $C_3\times C_3$}
\begin{center}
\scalebox{0.8}{
\begin{tabular}{ccccccc}
\hline
Const & $(\gamma_1,\gamma_2)$& $v_1$ & $(\gamma_3,\gamma_4)$&$v_2$ & $|Aut(C)|$ &Type \\ \hline \hline
$C_9$ &$(0,0)$& $(0,0,0,0,0,0,0,1,1)$ & $(0,1)$&$(0,0,1,1,1,0,1,1,1)$ & $2^{11} \cdot 3^2$ & $[40,20,8]_{I}$ \\ \hline
$C_9$ &$(0,0)$& $(0,0,0,0,1,0,1,1,1)$ & $(0,1)$&$(0,0,1,0,1,0,0,1,1)$ & $2^2 \cdot 3^2$ & $[40,20,8]_{I}$ \\ \hline
$C_9$ &$(0,0)$& $(0,0,0,1,1,1,1,1,1)$ & $(0,1)$&$(0,0,1,1,0,1,1,1,1)$ & $2^2 \cdot 3^2$ & $[40,20,8]_{I}$ \\ \hline
$C_9$ &$(1,0)$& $(0,0,0,0,0,0,1,1,1)$ & $(1,1)$&$(0,0,0,0,1,0,0,1,1)$ & $2^2 \cdot 3^2$ & $[40,20,8]_{II}$ \\ \hline
$C_9$ &$(1,0)$& $(0,0,0,0,0,0,1,1,1)$ & $(1,1)$&$(0,1,0,1,1,1,1,1,1)$ & $2^{11} \cdot 3^2$ & $[40,20,8]_{II}$ \\ \hline
$C_9$ &$(1,0)$& $(0,0,0,1,0,1,1,1,1)$ & $(1,1)$&$(0,0,1,1,0,1,0,1,1)$ & $2^3 \cdot 3^2 \cdot 5 \cdot 19$ & $[40,20,8]_{II}$ \\ \hline
$C_{3,3}$ & $(0,0)$& $(0,0,0,0,0,1,0,0,1)$ & $(0,1)$&$(0,1,1,0,1,1,1,0,1)$ & $2^{11} \cdot  3^2$ & $[40,20,8]_{I}$ \\ \hline
$C_{3,3}$ &$(0,0)$& $(0,0,0,0,1,1,0,1,1)$ & $(0,1)$&$(0,0,1,0,1,1,1,0,0)$ & $2^2 \cdot 3^2$ & $[40,20,8]_{I}$ \\ \hline
$C_{3,3}$ &$(0,0)$& $(0,0,1,0,1,1,1,1,1)$ & $(0,1)$&$(0,1,1,1,0,1,0,1,1)$ & $ 2^2 \cdot 3^2$ & $[40,20,8]_{I}$ \\ \hline
$C_{3,3}$ &$(1,0)$& $(0,0,0,0,0,1,0,1,1)$ & $(1,1)$&$(0,0,1,0,1,0,0,0,1)$ & $2^2 \cdot 3^2$ & $[40,20,8]_{II}$ \\ \hline
$C_{3,3}$ &$(1,0)$& $(0,0,1,0,0,1,0,0,1)$ & $(1,1)$&$(0,1,1,1,1,0,1,1,1)$ & $2^{11} \cdot 3^2$ & $[40,20,8]_{II}$ \\ \hline
$C_{3,3}$ &$(1,0)$& $(0,0,1,0,0,1,1,1,1)$ & $(1,1)$&$(0,0,1,0,1,1,1,0,1)$ & $2^3 \cdot 3^2 \cdot 5 \cdot 19$ & $[40,20,8]_{II}$ \\ \hline
$C_3\times C_3$&$(0,0)$& $(0,0,0,0,1,1,0,1,1)$ & $(0,1)$&$(0,0,1,0,0,1,1,1,0)$ & $2^{15} \cdot 3^2 \cdot 5$ & $[40,20,8]_{I}$ \\ \hline
$C_3\times C_3$&$(1,0)$& $(0,0,0,0,0,0,1,1,1)$ & $(1,1)$&$(0,0,1,0,0,1,0,1,0)$ & $2^4 \cdot 3^4$ & $[40,20,8]_{II}$ \\ \hline
$C_3\times C_3$&$(1,0)$& $(0,0,1,0,0,1,1,1,1)$ & $(1,1)$&$(0,0,1,1,1,0,1,1,0)$ & $2^{15} \cdot 3^2 \cdot 5$ & $[40,20,8]_{II}$ \\ \hline
\end{tabular}}%
\end{center}
\end{table}

The possible weight enumerators for a self-dual Type I $\left[80,40,14\right]$-code is given in \cite{Yankov} as:
\[ W_{80,2} =1+\left( 3200+4 \alpha \right) y^{14}+\left( 47645-8 \alpha+256 \beta \right) y^{16}+\cdots ,\]
where $\alpha$ and $\beta$ are integers. A $[80,40,14]$ was constructed in \cite{Dorfer} (it's weight enumerator was not stated) and a $[80,40,14]$ code was constructed in \cite{GULLHAR5}
with $\alpha=-280$, $\beta=10$. In \cite{Yankov}, $[80,40,14]$ codes were constructed for $\beta=0$ and $\alpha=-17k$ where $k \in \{2,\ldots,25,27\}$.

\begin{table}[H]
\caption{Binary $[80,40,14]$-codes from self-dual codes over $\FF_2+u\FF_2$ via $C_{9}$, $C_{3,3}$, $C_3\times C_3$.}
\begin{center}
\scalebox{0.8}{
\begin{tabular}{ccccccc}
\hline
 Const & $(\gamma_1,\gamma_2)$ & $v_1$ & $(\gamma_3,\gamma_4)$ & $v_2$ & $|Aut(\mathcal{C})|$ & $W_{80,2}$\\ \hline
 $C_9$ & $(u,u)$ & $(u,u,0,u,u,0,u,1,1)$ & $(0,1)$ & $(0,0,1,1,1,0,3,1,3)$ & $2^2\cdot3^2$ & $\alpha=-330$, $\beta=10$ \\ \hline
 $C_9$ & $(1,u)$ & $(u,0,0,3,u,1,3,3,3)$ & $(1,1)$ & $(u,0,1,3,0,3,u,1,1)$ & $2^2\cdot3^2$ & $\alpha=-258$, $\beta=1$ \\ \hline
 $C_9$ & $(u,0)$ & $(0,0,0,u,1,u,3,3,3)$ & $(u,1)$ & $(0,u,1,0,1,u,0,1,3)$ & $2^2\cdot3^2$ & $\alpha=-240$, $\beta=1$ \\ \hline
 $C_9$ & $(u,u)$ & $(u,0,0,0,1,0,1,3,3)$ & $(0,1)$ & $(u,u,1,u,3,0,u,3,1)$ & $2^2\cdot3^2$ & $\alpha=-204$, $\beta=1$ \\ \hline
 $C_9$ & $(u,u)$ & $(0,0,0,u,1,0,3,3,1)$ & $(0,1)$ & $(0,u,1,u,3,0,0,3,1)$ & $2^2\cdot3^2$ & $\alpha=-186$, $\beta=1$ \\ \hline
 $C_9$ & $(u,0)$ & $(u,u,0,u,1,0,1,1,1)$ & $(u,1)$ & $(u,u,1,u,1,u,u,3,3)$ & $2^3\cdot3^2$ & $\alpha=-168$, $\beta=1$ \\ \hline
 $C_9$ & $(u,0)$ & $(0,0,0,u,1,0,3,3,1)$ & $(u,1)$ & $(u,u,1,0,3,0,u,3,1)$ & $2^2\cdot3^2$ & $\alpha=-150$, $\beta=1$ \\ \hline
 $C_9$ & $(u,u)$ & $(u,u,0,u,1,0,1,1,1)$ & $(0,1)$ & $(0,0,1,0,1,0,0,3,3)$ & $2^3\cdot3^2$ & $\alpha=-96$, $\beta=1$ \\ \hline
 $C_{3,3}$&$(u,0)$ & $(u,u,u,u,u,1,0,0,1)$ & $(u,1)$ & $(u,1,1,u,3,1,3,u,1)$ & $2^2\cdot3^2$ & $\alpha=-366$, $\beta=10$ \\ \hline
 $C_{3,3}$&$(u,0)$ & $(u,u,u,u,u,1,u,0,3)$ & $(u,1)$ & $(u,1,3,0,1,3,1,u,3)$ & $2^2\cdot3^2$ & $\alpha=-348$, $\beta=10$ \\ \hline
 $C_{3,3}$&$(1,u)$ & $(0,u,u,0,u,3,u,3,1)$ & $(1,1)$ & $(u,u,1,u,3,u,u,u,3)$ & $2^3\cdot3^2$ & $\alpha=-312$, $\beta=1$ \\ \hline
 $C_{3,3}$&$(0,u)$ & $(0,0,0,u,u,1,0,0,1)$ & $(u,1)$ & $(u,1,1,u,3,1,3,u,1)$ & $2^2\cdot3^2$ & $\alpha=-294$, $\beta=10$ \\ \hline
 $C_{3,3}$&$(1,u)$ & $(u,0,u,0,u,3,u,1,3)$ & $(1,1)$ & $(0,0,3,u,1,u,u,u,3)$ & $2^2\cdot3^2$ & $\alpha=-222$, $\beta=1$ \\ \hline
 $C_{3,3}$&$(1,u)$ & $(0,0,u,0,u,3,0,3,1)$ & $(1,1)$ & $(0,0,3,u,1,u,u,u,3)$ & $2^2\cdot3^2$ & $\alpha=-168$, $\beta=1$ \\ \hline
 $C_{3,3}$&$(0,u)$ & $(0,0,u,u,1,1,0,3,3)$ & $(u,1)$ & $(0,0,1,0,3,1,3,u,0)$ & $2^2\cdot3^2$ & $\alpha=-186$, $\beta=1$ \\ \hline
 $C_3\times C_3$& $(u,u)$ & $(0,u,0,0,1,1,0,3,3)$ & $(1,1)$ & $(u,u,1,u,0,3,3,1,u)$ & $2^2\cdot3^2$ & $\alpha=-276$, $\beta=10$ \\ \hline
 $C_3\times C_3$& $(1,u)$ & $(u,u,3,0,0,3,1,3,3)$ & $(1,1)$ & $(u,0,3,3,3,u,1,3,0)$ & $2^3\cdot3^2$ & $\alpha=-276$, $\beta=10$ \\ \hline
 $C_3\times C_3$& $(1,u)$ & $(u,u,u,u,0,0,1,1,1)$ & $(1,1)$ & $(u,0,1,u,0,1,0,1,0)$ & $2^3\cdot3^2$ & $\alpha=-240$, $\beta=1$ \\ \hline
 $C_3\times C_3$& $(1,u)$ & $(u,u,3,0,0,3,1,3,3)$ & $(1,1)$ & $(u,0,3,3,3,0,3,1,u)$ & $2^3\cdot3^2$ & $\alpha=-204$, $\beta=10$ \\ \hline
 \end{tabular}}
\end{center}
\end{table}

\subsection{Constructions coming from $C_{13}$}
The best known Type I binary codes of length 56 have minimum weight 10. The possible weight enumerators for such a $\left[56,28,10\right]$-code is given in \cite{harada1} as:
\begin{eqnarray*}
W_{56,1} &=&1+\left( 308+4 \alpha \right) y^{10}+\left( 4246-8 \alpha \right) y^{12}+\left( 40852-28 \alpha \right) y^{14}+\cdots ,\\
W_{56,2} &=&1+\left( 308+4 \alpha \right) y^{10}+\left( 3990-8 \alpha \right) y^{12}+\left( 42900-28 \alpha \right) y^{14}+\cdots \\
\end{eqnarray*}%
where $\alpha$ is an integer. In \cite{harada1} codes constructed for the values of $\alpha=-18, -22, -24$ in $W_{56,1}$ and $\alpha=0, -2, -4, -6, -8, -10, -12, -14, -16, -18, -20, -22$ and $-24$ in
$W_{56,2}$.

In the following table we give a list of Type I self-dual codes of parameters $[56,28,10]$ by applying the construction $C_{13}$ over the binary field. The codes listed below all have new weight enumerators.

\begin{table}[H]
\caption{$[56,28,10]$ codes over $\FF_2$ from $C_{13}$.}
\begin{center}
\scalebox{0.8}{
\begin{tabular}{ccccccc}
\hline
 $(\gamma_1,\gamma_2)$& $v_1$ & $(\gamma_3,\gamma_4)$&$v_2$ & $|Aut(C)|$ & $\alpha$ & $W_{56,i}$\\ \hline \hline
 $(0,0)$ & $(0,0,0,0,0,0,0,1,0,1,0,1,1)$ & $(0,1)$ & $(0,0,0,0,1,1,1,0,1,1,1,1,1)$ & $2 \cdot 13$ & $-51$ & $1$\\ \hline
 $(0,0)$ & $(0,0,0,0,1,1,0,1,1,1,1,1,1)$ & $(0,1)$ & $(0,1,0,1,0,1,1,0,1,0,1,1,1)$ & $2 \cdot 13$ & $-38$ & $1$\\ \hline
 $(0,0)$ & $(0,0,0,0,0,0,1,1,0,1,1,1,1)$ & $(0,1)$ & $(0,0,0,0,1,1,0,1,0,0,1,1,1)$ & $2 \cdot 13$ & $-25$ & $1$\\ \hline
 $(0,0)$ & $(0,0,0,0,0,0,0,0,0,1,1,1,1)$ & $(0,1)$ & $(0,0,1,1,0,1,0,1,0,1,1,1,1)$ & $2^2 \cdot 13$ & $-38$ & $1$\\ \hline
 $(0,0)$ & $(0,0,0,0,0,0,0,0,0,0,0,1,1)$ & $(0,1)$ & $(0,0,0,1,0,0,1,0,1,1,1,0,1)$ & $2^2 \cdot 13$ & $-12$ & $1$\\ \hline
 $(0,0)$ & $(0,0,0,0,1,0,0,1,1,0,1,1,1)$ & $(0,1)$ & $(0,0,0,0,1,1,0,1,0,1,0,1,1)$ & $2^2 \cdot 3 \cdot 13$ & $-38$ & $1$\\ \hline
 $(0,0)$ & $(0,0,1,1,0,1,1,1,1,1,1,1,1)$ & $(0,1)$ & $(0,1,0,1,1,1,1,0,1,1,1,1,1)$ & $2^2 \cdot 3 \cdot 13$ & $-64$ & $1$\\ \hline
\end{tabular}}%
\end{center}
\end{table}

\subsection{Constructions coming from $C_{15}$}
We apply the constructions $C_{15}$ and $C_3\times C_5$ over the binary field to obtain a number of extremal binary self-dual codes of length 64. We only tabulate the codes obtained from construction $C_{15}$, however we note that we have obtained the exact same codes from $C_3\times C_5$ as well.

\begin{table}[H]
\caption{Binary self-dual codes of length $64$ from $C_{15}$.}
\begin{center}
\scalebox{0.7}{
\begin{tabular}{cccccc}
\hline
$(\gamma_1,\gamma_2)$& $v_1$ & $(\gamma_3,\gamma_4)$&$v_2$ & $|Aut(C)|$ & Type \\ \hline \hline
$(0,0)$& $(0,0,0,0,0,0,0,0,0,1,0,1,0,1,1)$ & $(0,1)$&$(0,0,0,0,1,0,1,0,1,0,0,0,1,1,1)$ & $2 \cdot 3 \cdot 5$ & $[64,32,12]_{II}$ \\ \hline
$(0,0)$& $(0,0,0,0,0,0,0,0,0,0,0,1,0,0,1)$ & $(0,1)$&$(0,0,0,0,1,1,0,1,1,0,0,1,1,1,1)$ & $2^2 \cdot 3 \cdot 5$ & $[64,32,12]_{II}$ \\ \hline
$(0,0)$& $(0,0,0,0,1,1,0,1,1,0,0,1,1,1,1)$ & $(0,1)$&$(0,0,0,1,0,0,0,1,0,0,0,1,1,1,1)$ & $2^3 \cdot 3 \cdot 5$ & $[64,32,12]_{II}$ \\ \hline
$(0,0)$& $(0,0,0,0,1,0,0,1,1,0,1,0,0,1,1)$ & $(0,1)$&$(0,0,0,1,1,0,1,0,1,1,0,1,0,1,1)$ & $2^{12} \cdot 3 \cdot 5$ & $[64,32,12]_{II}$ \\ \hline
$(1,0)$& $(0,0,0,0,0,0,0,0,0,1,0,1,1,1,1)$ & $(1,1)$&$(0,0,0,1,0,0,1,0,1,0,1,1,1,0,1)$ & $2 \cdot 3 \cdot 5$ & $W_{64,1}$ ($\beta=14$) \\ \hline
$(1,0)$& $(0,0,0,0,0,0,0,1,0,1,1,0,0,1,1)$ & $(1,1)$&$(0,0,0,0,1,0,0,1,1,0,1,1,0,1,1)$ & $2^2 \cdot 3 \cdot 5$ & $W_{64,1}$ ($\beta=14$) \\ \hline
$(1,0)$& $(0,0,0,0,0,0,0,0,0,0,1,0,0,1,1)$ & $(1,1)$&$(0,0,0,0,1,1,1,0,1,1,1,0,1,1,1)$ & $2^3 \cdot 3 \cdot 5$ & $W_{64,1}$ ($\beta=14$) \\ \hline
$(1,0)$& $(0,0,0,0,0,0,0,0,0,1,0,1,1,1,1)$ & $(1,1)$&$(0,0,0,0,0,1,0,1,0,0,1,1,1,1,1)$ & $2 \cdot 3 \cdot 5$ & $W_{64,1}$ ($\beta=29$) \\ \hline
$(1,0)$& $(0,0,0,0,0,0,0,0,1,0,1,0,1,1,1)$ & $(1,1)$&$(0,0,0,1,0,1,0,0,0,1,0,1,1,1,1)$ & $2 \cdot 3 \cdot 5$ & $W_{64,1}$ ($\beta=44$) \\ \hline
$(1,0)$& $(0,0,0,0,0,0,0,0,0,0,0,0,1,1,1)$ & $(1,1)$&$(0,0,0,1,0,1,1,1,1,1,0,1,0,1,1)$ & $2^2 \cdot 3 \cdot 5$ & $W_{64,1}$ ($\beta=44$) \\ \hline
$(1,0)$& $(0,0,0,0,0,1,0,0,0,1,0,1,0,1,1)$ & $(1,1)$&$(0,0,1,1,1,1,0,1,1,0,1,1,1,1,1)$ & $2 \cdot 3 \cdot 5$ & $W_{64,1}$ ($\beta=59$) \\ \hline
$(1,0)$& $(0,0,0,0,0,0,0,0,0,0,1,0,0,1,1)$ & $(1,1)$&$(0,0,0,1,0,0,1,1,1,1,1,1,0,1,1)$ & $2^2 \cdot 3 \cdot 5$ & $W_{64,1}$ ($\beta=74$) \\ \hline
\end{tabular}}%
\end{center}

\end{table}

\section{New Codes of length $68$}

In this section, we construct new extremal self-dual codes of length $68$ by extending certain previously constructed codes of length $64$ (using Theorem \ref{extension}) from Table \ref{T:1}. In particular we use $\mathcal{C}_{11}$ that was listed in Table \ref{T:1}.

\subsection{New Codes of length $68$ from $(\FF_4+u\FF_4)C_4$}

The possible weight enumerator of a self-dual $\left[ 68,34,12\right] _{2}$%
-code is in one of the following forms by \cite{buyuklieva,harada}:
\begin{eqnarray*}
W_{68,1} &=&1+\left( 442+4\beta \right) y^{12}+\left( 10864-8\beta \right)
y^{14}+\cdots ,104\leq \beta \leq 1358, \\
W_{68,2} &=&1+\left( 442+4\beta \right) y^{12}+\left( 14960-8\beta
-256\gamma \right) y^{14}+\cdots
\end{eqnarray*}%
where $0\leq \gamma \leq 9$. Recently, Yankov et al. constructed the first
examples of codes with a weight enumerator for $\gamma =7$ in $W_{68,2}$.
In \cite{Us2}, \cite{Us3} and \cite{GKY}, more unknown $W_{68,2}$ codes were constructed.
Together with these, the existence of the codes in $W_{68,2}$ is known for;
\begin{eqnarray*}
\gamma  &=&0,\ \beta \in \left\{ 2m\left\vert m=0,7,11,14,17,21,\ldots
,99,102,105,110,119,136,165\right. \right\} \text{ or} \\
\beta  &\in &\left\{ 2m+1\left\vert m=3,5,8,10,15,16,17,18,20,\ldots
,82,87,93,94,101,104,110,115\right. \right\} ; \\
\gamma  &=&1,\ \beta \in \left\{ 2m|m=22,24,\ldots ,99,102\right\} \text{
or }\beta \in \left\{ 2m+1|m=24,\ldots ,85,87\right\} ; \\
\gamma  &=&2,\ \beta \in \left\{ 2m\left\vert m=29,\ldots
,100,103,104\right. \right\} \text{ or }\beta \in \left\{ 2m+1\left\vert
m=32,34,\ldots ,79\right. \right\} ; \\
\gamma  &=&3,\ \beta \in \left\{ 2m\left\vert m=40,\ldots ,98,101,102\right.
\right\} \text{ or} \\
\beta  &\in &\left\{ 2m+1|m=41,43,\ldots ,77,79,80,83,96\right\} ; \\
\gamma  &=&4\text{, }\beta \in \left\{ 2m\left\vert m=43,44,48,\ldots
,92,97,98\right. \right\} \text{ or } \\
\beta  &\in &\left\{ 2m+1\left\vert m=48,\ldots ,55,58,60,\ldots
,78,80,83,84,85\right. \right\} ;\text{ } \\
\gamma  &=&5\text{ with }\beta \in \left\{ m|m=113,116,\ldots ,181\right\};
\\
\gamma  &=&6\text{ with }\beta \in \left\{ 2m|m=69,77,78,79,81,88\right\};  \\
\gamma  &=&7\text{ with }\beta \in \left\{ 7m|m=14,\ldots ,39,42\right\} .
\end{eqnarray*}

Recall that the previously constructed codes of length $64$ (from Table \ref{T:1}) are codes over $\FF_4+u\FF_4$. In order to apply
Theorem \ref{extension}, it requires the codes to be over $\FF_2+u\FF_2$. Before considering extensions of these
codes, we need to use the Gray map $\psi_{\FF_4+u\FF_4}$ to convert them to a code over $\FF_2+u\FF_2$. The following
table details the new extremal self-dual codes of length $68$. For each new code constructed we note the original code
of length $64$ from Table \ref{T:1}, the unit $c \in \FF_2+u\FF_2$, the vector $X$ required to apply Theorem \ref{extension}. The codes listed all have new weight enumerators.

\begin{table}[H]

\begin{center}

\caption{Type I Extremal Self-dual code of length $68$ from $C_{3}$ over $\FF_4+u\FF_4$.}\label{F4UF4C4}

\scalebox{0.7}{

\begin{tabular}{ccccccc} \hline

$\mathcal{C}_{68,i}$&  $\mathcal{C}_i$ & $c$  & $X$ & $\gamma$  &
$\beta$  \\ \hline
$\mathcal{C}_{68,1}$ &  $11$ &$1$ &
$(1,u,u,3,3,0,1,3,u,3,0,1,0,0,0,1,u,0,3,3,0,1,1,u,u,u,3,3,0,u,u,3)$
& $\textbf{4}$ & $\textbf{190}$
\\ \hline
$\mathcal{C}_{68,2}$ &  $11$ &$1$ &
$(0,1,0,1,3,1,0,0,u,u,1,u,u,0,1,1,1,0,u,1,u,1,1,0,1,0,3,3,u,1,u,u)$
& $\textbf{4}$ & $\textbf{192}$ \\ \hline
$\mathcal{C}_{68,3}$ &
$11$ &$u+1$ &
$(1,u,u,3,3,0,1,3,u,3,0,3,u,0,u,3,u,0,3,1,0,1,3,0,0,u,1,3,0,u,u,1)$
& $\textbf{4}$ & $\textbf{204}$ \\ \hline
$\mathcal{C}_{68,4}$ &
$11$ &$u+1$ &
$(u,1,0,3,0,0,0,u,1,u,u,0,0,0,3,3,1,3,u,0,0,u,3,1,0,0,u,0,0,0,1,3)$
& $\textbf{4}$ & $\textbf{208}$ \\ \hline
$\mathcal{C}_{68,5}$ &
$11$ &$1$ &
$(0,3,u,3,0,0,u,0,1,u,u,0,0,u,3,3,1,3,0,0,u,u,1,3,u,u,0,u,u,0,3,1)$
& $\textbf{4}$ & $\textbf{210}$ \\ \hline
$\mathcal{C}_{68,6}$ &
$11$ &$u+1$ &
$(u,1,u,1,0,0,u,0,3,0,u,0,0,u,1,1,3,3,0,0,0,u,3,3,0,0,0,0,u,0,1,1)$
& $\textbf{4}$ & $\textbf{214}$ \\ \hline
\end{tabular}}
  \end{center}
\end{table}

\subsection{New self-dual codes of length 68 from Neighboring construction}

Two self-dual binary codes of dimension $k$ are said to be neighbors
if their intersection has dimension $k-1$. Without loss of
generality, we consider the standard form of the
generator matrix of $C$. Let $x\in {\mathbb{F}}_{2}^{n}-C$ then $%
D=\left\langle \left\langle x\right\rangle ^{\bot }\cap
C,x\right\rangle $ is a neighbor of $C$. The first $34$ entries of
$x$ are set to be $0$, the rest of the vectors are listed in Table
\ref{neighbor}. As neighbors of codes in Table \ref{F4UF4C4} \ we
obtain $17$ new codes with weight enumerators in $W_{68,2},$ which
are listed in Table \ref{neighbor}. All the codes have an
automorphism group of order 2.

\begin{table}[H]
\caption{New codes of length 68 as neighbors of codes in Table \protect\ref%
{F4UF4C4} }
\label{neighbor}
\begin{center}
\scalebox{0.7}{
\begin{tabular}{ccccc||ccccc}
\hline
 $\mathcal{N}_{68,i}$ & $\mathcal{C}_{68,i}$ & $(x_{35},x_{36},...,x_{68})$ & $\gamma $ & $\beta $  &$\mathcal{N}_{68,i}$ & $\mathcal{C}_{68,i}$ & $(x_{35},x_{36},...,x_{68})$ & $\gamma $ & $\beta $\\ \hline
$1$ & $6$ & $(1111101001010000101110100001111010)$ & $3$ & $165$ & $2$ & $6$ & $(0011101000011001110010111010000011)$ & $3$ & $169$ \\ \hline
$3$ & $6$ & $(0110001110010110101000100011111101)$ & $3$ & $171$ & $4$ & $6$ & $(0100010010100011000110000110001010)$ & $3$ & $173$ \\ \hline
$5$ & $6$ & $(0110010001110000000011011110010100)$ & $4$ & $163$ & $6$ & $6$ & $(1110111111010101100001011001111011)$ & $4$ & $165$ \\ \hline
$7$ & $6$ & $(1000101111011011101011010101110100)$ & $4$ & $173$ & $8$ & $6$ & $(0100101010011010111001000111111100)$ & $4$ & $177$\\ \hline
$9$ & $6$ & $(1101110100111100110010000111001100)$ & $4$ & $179$ & $10$ & $6$ & $(1001010100010110110000010011000000)$ & $4$ & $181$ \\ \hline
$11$ & $2$ & $(1000101100010110000101111000010010)$ & $4$ & $183$ & $12$ & $6$ & $(0010111011011111100101111101000100)$ & $4$ & $185$ \\ \hline
$13$ & $6$ & $(1011011110010100011001011011001111)$ & $4$ & $187$ & $14$ & $6$ & $(0010010001110100011000001010000110)$ & $4$ & $188$\\ \hline
$15$ & $6$ & $(1001100011010110110101011110010001)$ & $4$ & $189$ & $16$ & $6$ & $(0111110011011110010101111010001100)$ & $4$ & $193$\\ \hline
$17$ & $5$ & $(0101101101011000110011101010001000)$ & $5$ & $201$ &  &  &  &  &   \\ \hline
\end{tabular}}%
\end{center}
\end{table}

\section{Conclusion}

In this work, we introduced a new construction for constructing self-dual codes using group rings. We provided certain conditions
when this construction produces self-dual codes and we established a link between units/non-units and self-dual codes. We demonstrated the relevance of this new construction by constructing many binary self-dual codes, including new self-dual codes of length $56$, $68$ and $80$.

\begin{itemize}
\item \textbf{Codes of length $56$:} We were to able to construct the following $[56,28,10]$ codes with new weight enumerators in $W_{56,1}$:
\[ \alpha \in \{-12,-25,-38,-51,-64\}.\]
\item \textbf{Codes of length $68$:} We were able to construct the following extremal binary self-dual codes with new weight enumerators in $W_{68,2}$:
\[
\begin{split}
(\gamma=3, &\quad \beta \in \{165,169,171,173\}),\\
(\gamma=4, &\quad \beta \in \{163,165,173,177,179,181,183,185,187,188,189,190,192,193,\\
& \qquad \quad 204,208,210,214\}),\\
(\gamma=5, &\quad \beta=201).
\end{split}
\]
The binary generator matrices of these codes are available online at \cite{web}.

\item \textbf{Codes of length $80$:} We were to able to construct the following $[80,40,14]$ codes with new weight enumerators in $W_{80,2}$:
\[
\begin{split}
(\beta=1, &\quad \alpha \in \{-96,-150,-168,-186,-204,-222,-240,-258,-312\}),\\
(\beta=10, &\quad \alpha \in \{-204,-276,-294,-330,-348,-366\}).
\end{split}
\]
\end{itemize}

We have considered specific lengths that are suitable for the constructions over some special alphabets. The number of extremal self-dual codes including many new ones that we have found through these conrstructions make us believe that this is a relevant direction for future research. A major direction direction of research could be considering other families of rings or groups of higher orders, which might require a considerable computational power.

\end{document}